\documentclass[10pt]{article}

\usepackage[margin=2.25cm]{geometry}

\usepackage{amsmath}
\usepackage{amssymb}
\usepackage{amsthm}
\usepackage{amsfonts}
\usepackage{enumerate}
\usepackage{graphicx}
\usepackage{url}
\usepackage[table]{xcolor}
\usepackage{soul}

\usepackage{subcaption}

\usepackage{ytableau}

\usepackage{tikz}
\usetikzlibrary{arrows,shapes,positioning,decorations.markings,decorations.pathreplacing}
\tikzset{->-/.style={decoration={
  markings,
  mark=at position .5 with {\arrow{>}}},postaction={decorate}}}

\usepackage{array}
\newcolumntype{C}[1]{>{\centering\let\newline\\\arraybackslash\hspace{0pt}}m{#1}}

\newtheorem{theorem}{Theorem}[section]

\newtheorem{lemma}[theorem]{Lemma}
\newtheorem{example}[theorem]{Example}
\newtheorem{corollary}[theorem]{Corollary}
\newtheorem{proposition}[theorem]{Proposition}

\newcommand\floor[1]{\lfloor#1\rfloor}

\DeclareMathOperator{\svt}{\mathbb{S}}
\DeclareMathOperator{\im}{im}

\begin{document}

\title{\bf Set-Valued Tableaux \& Generalized Catalan Numbers}
\author{Paul Drube\\
\small Department of Mathematics and Statistics\\[-0.8ex]
\small Valparaiso University\\[-0.8ex] 
\small Valparaiso, Indiana, U.S.A.\\
\small\tt paul.drube@valpo.edu\\
}

\maketitle

\begin{abstract}
Standard set-valued Young tableaux are a generalization of standard Young tableaux in which cells may contain more than one integer, with the added conditions that every integer at position $(i,j)$ must be smaller than every integer at positions $(i,j+1)$ and $(i+1,j)$.  This paper explores the combinatorics of standard set-valued Young tableaux with two-rows, and how those tableaux may be used to provide new combinatorial interpretations of generalized Catalan numbers.  New combinatorial interpretations are provided for the two-parameter Fuss-Catalan numbers (Raney numbers), the rational Catalan numbers, and the solution to the so-called ``generalized tennis ball problem".  Methodologies are then introduced for the enumeration of standard set-valued Young tableaux, prompting explicit formulas for the general two-row case.  The paper closes by drawing a bijection between arbitrary classes of two-row standard set-valued Young tableaux and collections of two-dimensional lattice paths that lie weakly below a unique maximal path.

\end{abstract}

\section{Introduction}
\label{sec: intro}

For a non-increasing integer partition $\lambda = (\lambda_1,\lambda_2,\hdots,\lambda_m)$, a Young diagram $Y$ of shape $\lambda$ is a left-justified array of cells with exactly $\lambda_i$ cells in its $i^{th}$ row.  If $Y$ is a Young diagram of shape $\lambda$ with $\sum_i \lambda_i = n$, a Young tableau of shape $\lambda$ is an assignment of the integers $[n] = \lbrace 1,\hdots,n \rbrace$ to the cells of $Y$ such that every integer is used precisely once.  A Young tableau in which integers increase from top-to-bottom down every column and increase from left-to-right across every row is said to be a standard Young tableau.  We denote the set of all standard Young tableaux of shape $\lambda$ by $S(\lambda)$.  For $m$-row rectangular shapes $\lambda = (n,\hdots,n)$ we use the abbreviated notation $S(n^m)$.  For a comprehensive introduction to Young tableaux, see Fulton \cite{Fulton}.

Let $Y$ be a Young diagram of shape $\lambda$, and let $\rho = \lbrace \rho_{i,j} \rbrace$ be a collection of positive integers such that $\sum_{i,j} \rho_{i,j} = m$.  A \textbf{set-valued tableau} of shape $\lambda$ and density $\rho$ is a function from $[m]$ to the cells of $Y$ such that the cell at position $(i,j)$ receives a set of $\rho_{i,j}$ integers.  A set-valued tableau is said to be a \textbf{standard set-valued Young tableau} if we additionally require that every integer at position $(i,j)$ is smaller than every integer at positions $(i+1,j)$ and $(i,j+1)$.  In analogy with standard Young tableaux, we refer to these added conditions as ``column-standardness" and ``row-standardness".  We denote the set of all standard set-valued Young tableaux of shape $\lambda$ and density $\rho$ by $\svt(\lambda,\rho)$.  See Figure \ref{fig: set-valued tableaux basic example} for a basic example.

\begin{figure}[ht!]
\centering
\begin{ytableau}
1 \kern+2pt 2 & 3 \kern+2pt 4 \\
5 \kern+2pt 6 & 7 \kern+2pt 8
\end{ytableau}
\hspace{.3in}
\begin{ytableau}
1 \kern+2pt 2 & 3 \kern+2pt 5 \\
4 \kern+2pt 6 & 7 \kern+2pt 8
\end{ytableau}
\hspace{.3in}
\begin{ytableau}
1 \kern+2pt 2 & 3 \kern+2pt 6 \\
4 \kern+2pt 5 & 7 \kern+2pt 8
\end{ytableau}
\hspace{.3in}
\begin{ytableau}
1 \kern+2pt 2 & 4 \kern+2pt 5 \\
3 \kern+2pt 6 & 7 \kern+2pt 8
\end{ytableau}
\hspace{.3in}
\begin{ytableau}
1 \kern+2pt 2 & 4 \kern+2pt 6 \\
3 \kern+2pt 5 & 7 \kern+2pt 8
\end{ytableau}
\hspace{.3in}
\begin{ytableau}
1 \kern+2pt 2 & 5 \kern+2pt 6 \\
3 \kern+2pt 4 & 7  \kern+2pt  8
\end{ytableau}

\caption{The set $\svt(\lambda,\rho)$ when $\lambda = (2,2)$ and $\rho_{i,j} = 2$ for all $i,j$.}
\label{fig: set-valued tableaux basic example}
\end{figure}

Set-valued tableaux were introduced by Buch \cite{Buch} in his investigation of the K-theory of Grassmannians.  More directly influencing this paper is the work of Heubach, Li and Mansour \cite{HLM}, who argued that the cardinality of $\svt(n^2,\rho)$ with row-constant density $\rho_{1,j}=k-1$, $\rho_{2,j}=1$ equalled the $k$-Catalan number $C_n^k$.  For a more recent appearance of standard set-valued tableaux see Reiner, Tenner and Yong \cite{RTY}, who investigated so-called ``barely set-valued tableaux" with a single non-unitary density $\rho_{i,j} = 2$ (not necessarily located at a fixed position $i,j$).

Currently, the central difficulty in studying standard set-valued Young tableaux is the lack of a closed formula for enumerating general $\svt(\lambda,\rho)$: there is no known set-valued analogue of the celebrated hook-length formula for standard Young tableaux.  Reiner, Tenner and Yong \cite{RTY} do utilize a modified insertion algorithm to enumerate their ``barely set-valued tableaux", but theirs is an atypically tractable case and cannot be modified to the enumeration of sets $\svt(\lambda,\rho)$ with a fixed density at each position.

The purpose of this paper is twofold.  In Section \ref{sec: generalized Catalan numbers}, we utilize standard set-valued Young tableaux to provide new combinatorial interpretations for various generalizations of the Catalan numbers.  In particular, we show how various densities for rectangular set-valued tableaux of shape $\lambda=n^2$ are enumerated by the Raney numbers (two-parameter Fuss Catalan numbers, Theorem \ref{thm: Raney numbers}) and the solution to the "$(s,t)$-tennis ball problem" of Merlini, Sprugnoli, and Verri \cite{MSV} (Theorem \ref{thm: tennis ball problem}).  See Figure \ref{fig: weights preview} for an overview of the various densities needed to achieve our combinatorial interpretations.  In Section \ref{sec: enumeration}, we develop methodologies for enumerating two-row standard set-valued tableaux.  Concise closed formulas are presented for the number of such tableaux of arbitrary density (Theorem \ref{thm: repeated density shifting}).  In Section \ref{sec: lattice paths}, we conclude by drawing a bijection between two-row standard set-valued tableaux of arbitrary density and certain classes of two-dimensional integer lattice paths with ``East" $E=(1,0)$ and ``North" $N=(0,1)$.  In particular, $\svt(n^2,\rho)$ with $\rho_{1,j} = a_j$ and $\rho_{2,j} = b_j$ is placed in bijection with all such lattice paths that lie weakly below the lattice path $P = E^{a_1} N^{b_1} E^{a_2} N^{b_2} \hdots$ (Theorem \ref{thm: tableaux vs lattice paths}).

\begin{figure}[ht!]
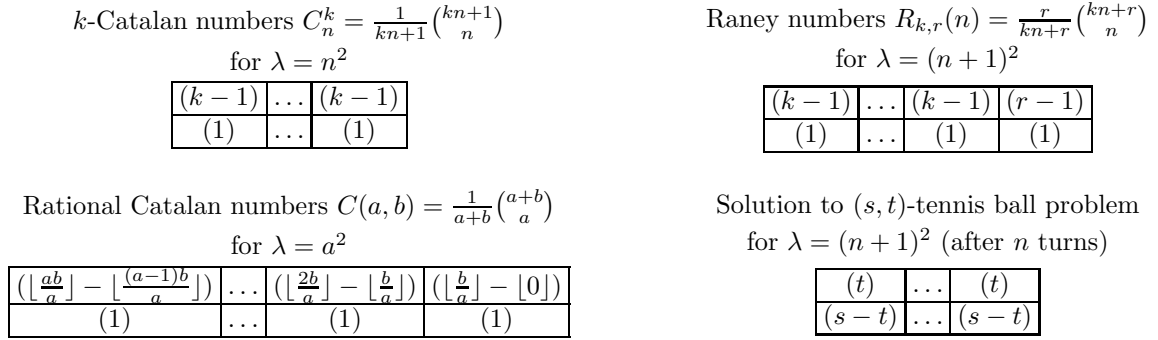

\begin{subfigure}[b]{0.58\textwidth}
\centering
$k$-Catalan numbers $C_n^k = \frac{1}{kn + 1} \binom{kn+1}{n}$\\
\vspace{.025in}
for $\lambda = n^2$\\
\vspace{.05in}
\setlength{\tabcolsep}{2.5pt}
\begin{tabular}{|>{$}c<{$}|>{$}c<{$}|>{$}c<{$}|}
\hline
(k-1) & \hdots & (k-1) \\ \hline
(1) & \hdots & (1) \\ \hline
\end{tabular}\\
\vspace{.2in}
Rational Catalan numbers $C(a,b) = \frac{1}{a+b} \binom{a+b}{a}$\\
\vspace{.025in}
for $\lambda = a^2$\\
\vspace{.05in}
\setlength{\tabcolsep}{2.5pt}
\begin{tabular}{|>{$}c<{$}|>{$}c<{$}|>{$}c<{$}|>{$}c<{$}|}
\hline
(\floor{\frac{ab}{a}} - \floor{\frac{(a-1)b}{a}}) & \hdots & (\floor{\frac{2b}{a}} - \floor{\frac{b}{a}}) & (\floor{\frac{b}{a}} - \floor{0}) \\ \hline
(1) & \hdots & (1) & (1) \\ \hline
\end{tabular}
\end{subfigure}
\begin{subfigure}[b]{0.40\textwidth}
\centering
Raney numbers $R_{k,r}(n) = \frac{r}{kn+r} \binom{kn+r}{n}$\\
\vspace{.025in}
for $\lambda = (n+1)^2$\\
\vspace{.05in}
\setlength{\tabcolsep}{2.5pt}
\begin{tabular}{|>{$}c<{$}|>{$}c<{$}|>{$}c<{$}|>{$}c<{$}|}
\hline
(k-1) & \hdots & (k-1) & (r-1)\\ \hline
(1) & \hdots & (1) & (1) \\ \hline
\end{tabular}\\
\vspace{.2in}
Solution to $(s,t)$-tennis ball problem \\
\vspace{.025in}
for $\lambda = (n+1)^2$ (after $n$ turns)\\
\vspace{.05in}
\setlength{\tabcolsep}{2.5pt}
\begin{tabular}{|>{$}c<{$}|>{$}c<{$}|>{$}c<{$}|}
\hline
(t) & \hdots & (t) \\ \hline
(s-t) & \hdots & (s-t) \\ \hline
\end{tabular}
\end{subfigure}
\caption{Densities $\rho$ needed for $\vert \svt(\lambda,\rho) \vert$ to yield various combinatorial interpretations.}
\label{fig: weights preview}
\end{figure}

We pause to introduce a foundational result that, in the case of rectangular $\lambda$, serves as a set-valued analogue of the Sch\"utzeberger involution for standard Young tableaux.  In the case of densities $\rho$ that are constant across each row, notice that Proposition \ref{thm: Schutzenberger involution} manifests as invariance under a vertical reflection of those densities.

\begin{proposition}
\label{thm: Schutzenberger involution}
For rectangular $\lambda = n^m$ and any density $\rho = \lbrace \rho_{i,j} \rbrace$, let $r(\rho) = \lbrace \rho_{n-i+1,m-j+1} \rbrace$.  Then $\vert \svt(\lambda,\rho) \vert = \vert \svt(\lambda,r(\rho)) \vert$.
\end{proposition}
\begin{proof}
One may define a bijection $f : S(\lambda,\rho) \rightarrow S(\lambda,r(\rho))$ such that $f(T) \in S(\lambda,r(\rho))$ is obtained by reversing the alphabet of $T \in S(\lambda,\rho)$ and rotating the resulting tableau by 180-degrees.
\end{proof}

\section{Generalized Catalan Numbers and Set-Valued Tableaux}
\label{sec: generalized Catalan numbers}

We begin be briefly recapping established results about the $k$-Catalan numbers.  For any $k \geq 1$, the $k$-Catalan numbers are given by $C_n^k = \frac{1}{kn+1} \binom{kn+1}{n}$ for all $n \geq 0$.  Notice that the $k$-Catalan numbers specialize to the usual Catalan numbers when $k = 2$.

See Hilton and Pedersen \cite{HP} or Heubach, Li and Mansour \cite{HLM} for various combinatorial interpretations of the $k$-Catalan numbers.  Relevant to our work is the standard result that $C_n^k$ enumerates the set $\mathcal{D}_n^k$ of so-called $k$-good paths of length $kn$.  These are integer lattice paths from $(0,0)$ to $(n,(k-1)n)$ that utilize only ``East" $E=(1,0)$ and ``North" $N = (0,1)$ steps and which stay weakly below the line $y = (k-1)x$.  The set of $k$-good paths are obviously in bijection with $k$-ary paths of length $kn$: integer lattice paths from $(0,0)$ to $(nk,0)$ that use steps $u=(1,\frac{1}{k-1})$, $d=(1,-1)$ and stay weakly above $y=0$.  We prefer working with $k$-good paths because they are more easily generalized to sets enumerated by the Raney numbers and the rational Catalan numbers. 

Heubach, Li and Mansour \cite{HLM} showed that standard set-valued Young tableaux of shape $\lambda = n^2$ and row-constant density $\rho_{1,j} = k-1$, $\rho_{2,j} = 1$ were counted by $C_n^k$.  This was done by placing such tableaux in bijection with $k$-ary paths of length $kn$.  Using Proposition \ref{thm: Schutzenberger involution}, their map may be modified to give an elementary bijection between $\mathcal{D}_n^k$ and standard set-valued Young tableaux $S(n^2,\rho)$ of row-constant density $\rho_{1,j} = 1$, $\rho_{2,j} = k-1$.  As exemplified in Figure \ref{fig: k-Catalan interpretations}, this bijection $\phi: \mathcal{D}_n^k \rightarrow S(\lambda,w)$ involves associating East steps in $P \in \mathcal{D}_n^k$ with first-row entries in $\phi(P)$ and North steps in $P$ with second-row entries in $\phi(P)$.

\begin{figure}[ht!]
\centering
\raisebox{53pt}{\begin{ytableau}
1 & 3 & 7 \\
2 \kern+2pt 4 & 5 \kern+2pt 6 & 8 \kern+2pt 9
\end{ytableau}}
\hspace{.2in}
\raisebox{44pt}{\scalebox{3}{$\Leftrightarrow$}}
\hspace{.2in}
\scalebox{.9}{
\begin{tikzpicture}
	[scale=.6,auto=left,every node/.style={circle, fill=black,inner sep=1.2pt}]
	\draw[dotted] (0,0) to (3,0);
	\draw[dotted] (0,1) to (3,1);
	\draw[dotted] (0,2) to (3,2);
	\draw[dotted] (0,3) to (3,3);
	\draw[dotted] (0,4) to (3,4);
	\draw[dotted] (0,5) to (3,5);
	\draw[dotted] (0,6) to (3,6);
	\draw[dotted] (0,0) to (0,6);
	\draw[dotted] (1,0) to (1,6);
	\draw[dotted] (2,0) to (2,6);
	\draw[dotted] (3,0) to (3,6);
	\draw (0,0) to (3,6);
	\node (0*) at (0,0) {};
	\node (1*) at (1,0) {};
	\node (2*) at (1,1) {};
	\node (3*) at (2,1) {};
	\node (4*) at (2,2) {};
	\node (5*) at (2,3) {};
	\node (6*) at (2,4) {};
	\node (7*) at (3,4) {};
	\node (8*) at (3,5) {};
	\node (9*) at (3,6) {};
	\node[fill=none] (1l) at (.55,.25) {\scalebox{.85}{1}};
	\node[fill=none] (2l) at (1.25,.45) {\scalebox{.85}{2}};
	\node[fill=none] (3l) at (1.5,1.275) {\scalebox{.85}{3}};
	\node[fill=none] (4l) at (2.25,1.45) {\scalebox{.85}{4}};
	\node[fill=none] (5l) at (2.25,2.45) {\scalebox{.85}{5}};
	\node[fill=none] (6l) at (2.25,3.45) {\scalebox{.85}{6}};
	\node[fill=none] (7l) at (2.55,4.275) {\scalebox{.85}{7}};
	\node[fill=none] (8l) at (3.25,4.45) {\scalebox{.85}{8}};
	\node[fill=none] (9l) at (3.25,5.45) {\scalebox{.85}{9}};
	\draw[thick] (0*) to (1*);
	\draw[thick] (1*) to (2*);
	\draw[thick] (2*) to (3*);
	\draw[thick] (3*) to (4*);
	\draw[thick] (4*) to (5*);
	\draw[thick] (5*) to (6*);
	\draw[thick] (6*) to (7*);
	\draw[thick] (7*) to (8*);	
	\draw[thick] (8*) to (9*);
\end{tikzpicture}}
\caption{A standard set-valued Young tableau with $\lambda = 3^2$ and $\rho_{1,j} = 1,\rho_{2,j}=k-1$, alongside the corresponding $3$-good path in $\mathcal{D}_3^3$.}
\label{fig: k-Catalan interpretations}
\end{figure}
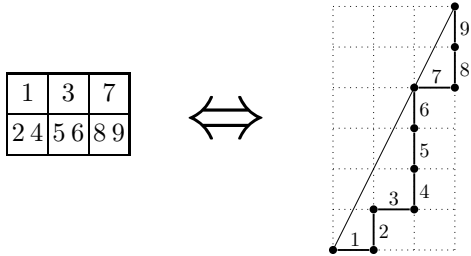

\subsection{Raney numbers}
\label{subsec: Raney numbers}

The first generalization of the Catalan numbers for which we will present a new combinatorial interpretation are the Raney numbers, also known as the two-parameter Fuss-Catalan numbers.  For any $k \geq 1$ and $r \geq 1$, the Raney numbers are given by $R_{k,r}(n) = \frac{r}{kn+r} \binom{kn+r}{n}$ for all $n \geq 0$.\footnote{Hilton and Pedersen \cite{HP} use the alternative notation $d_{qk}(p)=\frac{p-q}{pk-q} \binom{pk-q}{k-1}$ for their two-parameter generalization.  Our two notations are related via the change of variables $R_{p,p-q}(k-1) = d_{qk}(p)$.}  The Raney numbers specialize to the $k$-Catalan numbers as $R_{k,1}(n) = C_n^k$.

Hilton and Pedersen \cite{HP} showed that the Raney numbers could be calculated from the $k$-Catalan numbers via Equation \ref{eq: Raney recurrence}.  This equation may be viewed as a generalization of the standard recurrence for the $k$-Catalan numbers when one notes that $R_{k,k}(n-1) = C_n^k$, a distinct identity from the ``obvious" specialization of Equation \ref{eq: Raney recurrence} to the $k$-Catalan numbers as $R_{k,1}(n) = C_n^k$.

\begin{equation}
\label{eq: Raney recurrence}
R_{k,r}(n) \ = \kern-6pt \sum_{(i_1,\hdots,i_r) \vdash n} \kern-6pt C_{i_1}^k C_{i_2}^k \hdots C_{i_r}^k
\end{equation}

Equation \ref{eq: Raney recurrence} is useful in that it allows one to define combinatorial interpretations for the Raney numbers as ordered $r$-tuples of pre-existing interpretations for the $k$-Catalan numbers:

\begin{proposition}
\label{thm: Raney numbers lemma}
Fix $k,r \geq 1$, $n \geq 0$.  Then $R_{k,r}(n)$ equals the number of ordered $r$-tuples $(T_1,\hdots,T_r)$ such that $T_j \in \svt(i_j^2,\rho)$ for row-constant weight $\rho_{1,j} = 1$, $\rho_{2,j}=k-1$, where $i_1 + \hdots + i_r = n$.
\end{proposition}

Our goal is to replace the ordered $r$-tuples of Proposition \ref{thm: Raney numbers lemma} with a single set-valued tableau of shape $\lambda = (n+1)^2$.  This utilizes a technique that we refer to ``horizontal tableaux concatenation", whereby the entries of the ordered $r$-tuple are continuously reindexed and a new column with density $\rho_{1,1}=1,\rho_{2,1}=r-1$ is added to the front of the resulting tableau.  This additional column carries the information needed to recover the original partition of the tableau into $r$ pieces.

So fix $n \geq 0$ and take any two-row rectangular shape $\lambda = (n+1)^2$.  To ease notation, for any $k,r \geq 1$ we temporarily define the density $\rho(k,r) = \lbrace \rho_{i,j} \rbrace$ by $\rho_{1,j} = 1$ for all $1 \leq j \leq n$, $\rho_{2,1} = r-1$, and $\rho_{2,j} = k-1$ for all $2 \leq j \leq n$.  Notice that $\rho(k,r)$ is equivalent to the density of Figure \ref{fig: weights preview} via Proposition \ref{thm: Schutzenberger involution}.

\begin{theorem}
\label{thm: Raney numbers}
Take any $k,r \geq 1$, $n \geq 0$, and define $\rho(k,r)$ as above.  Then $R_{k,r}(n) = \vert \svt((n+1)^2,\rho(k,r)) \vert$.
\end{theorem}
\begin{proof}
Take $(T_1,\hdots,T_r)$, where $T_j \in \svt(i_j^2,\rho)$ with $\rho_{1,j}=1,\rho_{2,j}=k-1$ and $i_1 + \hdots + i_r = n$.  Observe that a total of $kn$ integers appear across the $2n$ cells of the $T_j$.  Create a partially-filled Young diagram $D$ of shape $\lambda_D = (n+r)^2$ by adding an empty column $c_j$ in front of each $T_j$ and horizontally concatenating the resulting tableaux in the given order.  Notice that this will result in multiple consecutive empty columns if any of the $T_j$ are empty.  Mark the top cell of column $c_1$ and the bottom cells of columns $c_2,\hdots,c_r$.  This gives $r$ markings in addition to the $kn$ integers of $D$.  Re-index these $kn+r$ items by working through $D$ from left-to-right.  Every time a marking is encountered, assign the marked cell the smallest available element of $[kn+r]$.  When $T_j$ is encountered, simultaneously replace the $k i_j$ integers of $T_j$ with the $k i_j$ smallest available elements of $[kn+r]$, preserving the relative ordering within $T_j$.

This gives a partially-filled set-valued Young tableau $\widetilde{D}$ that is row-standard and column-standard if you look past the empty cells.  Then ``collapse" the entries of $\widetilde{D}$ off the interstitial columns $c_2,\hdots,c_j$ (but not off $c_1$) by shifting all entries leftward until the cells corresponding to $c_1$ have density $\rho_{1,1}=1,\rho_{1,2} = r-1$, the cells corresponding to the $T_j$ have row-constant density $\rho_{1,j}=1,\rho_{2,j}=k-1$, and the cells corresponding to the columns $c_2,\hdots,c_j$ are empty.  Deleting the columns corresponding to $c_2,\hdots,c_j$ then produces a set-valued tableau $T$ of shape $\lambda = (n+1)^2$ and density $\rho(k,r)$.  $T$ is obviously row-standard.  To see that $T$ is also column-standard, notice that first row entries of $\widetilde{D}$ that were originally associated with a particular $T_j$ are shifted leftward by precisely $j-1$ cells as we pass from $\widetilde{D}$ to $T$, whereas second row entries in $\widetilde{D}$ that were originally associated with $T_j$ are shifted leftward by at least $j-1$ cells as we pass from $\widetilde{D}$ to $T$.  The latter observation follows from the fact that $r-1$ integers must eventually appear in the $(2,1)$ cell of $T$, and that there are $j-1 \leq r-1$ marked second-row cells to the left of the entries associated with $T_j$.  As second row entries are shifted at least as far left as first row entries, the column-standardness of $\widetilde{D}$ implies that $T$ is also column-standard.

To show that our map $(T_1,\hdots,T_r) \mapsto T$ is bijective we provide a well-defined inverse.  For $T \in \svt((n+1)^2,\rho(k,r))$, collectively shift all entries in the second row rightward so that the $(2,1)$ cell is empty and all remaining cells in the second row contain precisely $k-1$ entries (there will be an overflow of $r-1$ elements at the end of the second row).  Then proceed through the second-row from left-to-right and identify the smallest integer $c$ that violates column-standardness.  Insert a new, partially-filled column at the position of $c$ whose top cell is empty and whose bottom cell contains $c$.  Then re-allocate the remaining entries of the second row so that $k-1$ entries appear in each cell to the right of $c$, and repeat the above procedure until $r-1$ new columns have been added.  The end result of this procedure is identical to the partially-filled tableau $\widetilde{D}$ from above.  This is because the second-row entries of the interstitial columns $c_2,\hdots,c_r$ are necessarily smaller than all entries in the ``block" corresponding to $T_j$ and hence would violate column-standardness if moved one cell to their right.
\end{proof}

\begin{figure}[ht!]
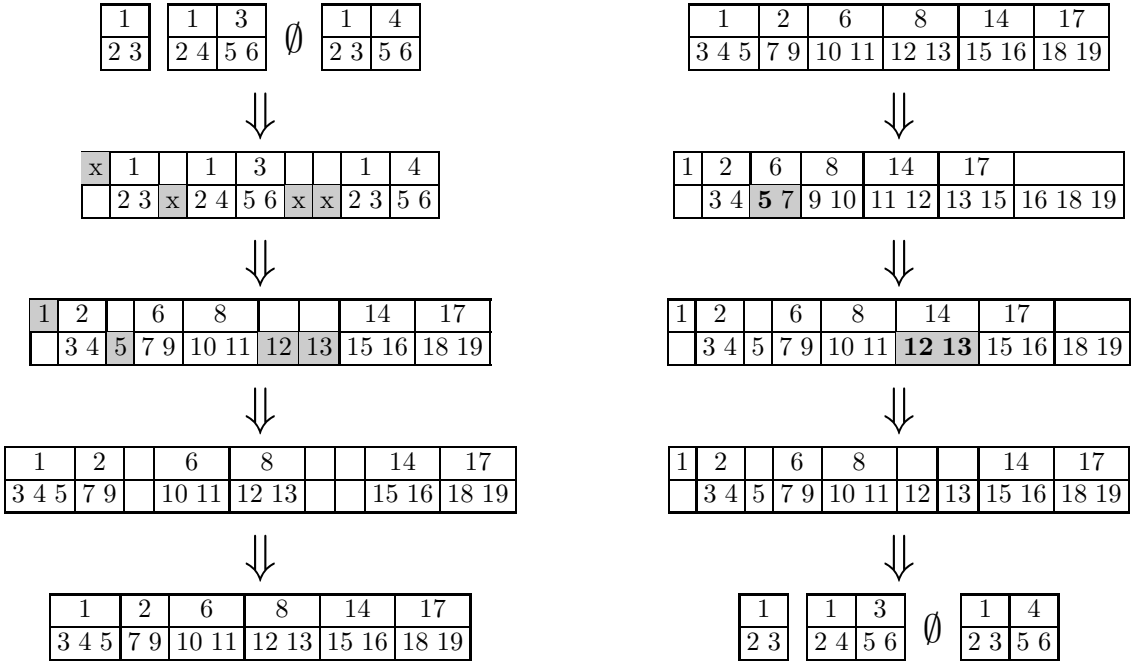

\begin{subfigure}[b]{0.49\textwidth}
\centering
\setlength{\tabcolsep}{2.5pt}
\begin{tabular}{|c|}
\hline
1 \\ \hline
2 \kern-.25pt 3 \\ \hline
\end{tabular}
\hspace{.0in}
\begin{tabular}{|c|c|}
\hline
1 & 3 \\ \hline
2 \kern-.25pt 4 & 5 \kern-.25pt 6 \\ \hline
\end{tabular}
\hspace{.0in}
\raisebox{-3pt}{\scalebox{1.5}{$\emptyset$}}
\hspace{.0in}
\begin{tabular}{|c|c|}
\hline
1 & 4 \\ \hline
2 \kern-.25pt 3 & 5 \kern-.25pt 6 \\ \hline
\end{tabular}

\vspace{.1in}

\scalebox{2}{$\Downarrow$}

\vspace{.05in}

\begin{tabular}{|c|c|c|c|c|c|c|c|c|}
\hline
\cellcolor{gray!40}x & 1 & & 1 & 3 & & & 1 & 4 \\ \hline
& 2 \kern-.25pt 3 & \cellcolor{gray!40}x & 2 \kern-.25pt 4 & 5 \kern-.25pt 6 & \cellcolor{gray!40}x & \cellcolor{gray!40}x & 2 \kern-.25pt 3 & 5 \kern-.25pt 6 \\ \hline
\end{tabular}

\vspace{.1in}

\scalebox{2}{$\Downarrow$}

\vspace{.05in}

\begin{tabular}{|c|c|c|c|c|c|c|c|c|}
\hline
\cellcolor{gray!40}1 & 2 & & 6 & 8 & & & 14 & 17 \\ \hline
& 3 \kern-.25pt 4 & \cellcolor{gray!40}5 & 7 \kern-.25pt 9 & 10 \kern-.25pt 11 & \cellcolor{gray!40}12 & \cellcolor{gray!40}13 & 15 \kern-.25pt 16 & 18 \kern-.25pt 19 \\ \hline
\end{tabular}

\vspace{.1in}

\scalebox{2}{$\Downarrow$}

\vspace{.05in}

\begin{tabular}{|c|c|c|c|c|c|c|c|c|}
\hline
1 & 2 & & 6 & 8 & & & 14 & 17 \\ \hline
3 \kern-.25pt 4 \kern-.25pt 5 & 7 \kern-.25pt 9 & \hspace{6pt} & 10 \kern-.25pt 11 & 12 \kern-.25pt 13 & \hspace{6pt} & \hspace{6pt} & 15 \kern-.25pt 16 & 18 \kern-.25pt 19 \\ \hline
\end{tabular}

\vspace{.1in}

\scalebox{2}{$\Downarrow$}

\vspace{.05in}

\begin{tabular}{|c|c|c|c|c|c|c|}
\hline
1 & 2 & 6 & 8 & 14 & 17 \\ \hline
3 \kern-.25pt 4 \kern-.25pt 5 & 7 \kern-.25pt 9 & 10 \kern-.25pt 11 & 12 \kern-.25pt 13 & 15 \kern-.25pt 16 & 18 \kern-.25pt 19 \\ \hline
\end{tabular}

\end{subfigure}
\begin{subfigure}[b]{0.49\textwidth}
\centering
\setlength{\tabcolsep}{2.5pt}
\begin{tabular}{|c|c|c|c|c|c|c|}
\hline
1 & 2 & 6 & 8 & 14 & 17 \\ \hline
3 \kern-.25pt 4 \kern-.25pt 5 & 7 \kern-.25pt 9 & 10 \kern-.25pt 11 & 12 \kern-.25pt 13 & 15 \kern-.25pt 16 & 18 \kern-.25pt 19 \\ \hline
\end{tabular}

\vspace{.1in}

\scalebox{2}{$\Downarrow$}

\vspace{.05in}

\begin{tabular}{|c|c|c|c|c|c|c|c|}
\hline
1 & 2 & 6 & 8 & 14 & 17 & \\ \hline
& 3 \kern-.25pt 4 & \cellcolor{gray!40}\textbf{5} \kern-.25pt 7 & 9 \kern-.25pt 10 & 11 \kern-.25pt 12 & 13 \kern-.25pt 15 & 16 \kern-.25pt 18 \kern-.25pt 19 \\ \hline
\end{tabular}

\vspace{.1in}

\scalebox{2}{$\Downarrow$}

\vspace{.05in}

\begin{tabular}{|c|c|c|c|c|c|c|c|c|}
\hline
1 & 2 & & 6 & 8 & 14 & 17 & \\ \hline
& 3 \kern-.25pt 4 & 5 & 7 \kern-.25pt 9 & 10 \kern-.25pt 11 & \cellcolor{gray!40}\textbf{12} \kern-.25pt \textbf{13} & 15 \kern-.25pt 16 & 18 \kern-.25pt 19 \\ \hline
\end{tabular}

\vspace{.1in}

\scalebox{2}{$\Downarrow$}

\vspace{.05in}

\begin{tabular}{|c|c|c|c|c|c|c|c|c|c|}
\hline
1 & 2 & & 6 & 8 & & & 14 & 17 \\ \hline
& 3 \kern-.25pt 4 & 5 & 7 \kern-.25pt 9 & 10 \kern-.25pt 11 & 12 & 13 & 15 \kern-.25pt 16 & 18 \kern-.25pt 19 \\ \hline
\end{tabular}

\vspace{.1in}

\scalebox{2}{$\Downarrow$}

\vspace{.05in}

\begin{tabular}{|c|}
\hline
1 \\ \hline
2 \kern-.25pt 3 \\ \hline
\end{tabular}
\hspace{.0in}
\begin{tabular}{|c|c|}
\hline
1 & 3 \\ \hline
2 \kern-.25pt 4 & 5 \kern-.25pt 6 \\ \hline
\end{tabular}
\hspace{.0in}
\raisebox{-3pt}{\scalebox{1.5}{$\emptyset$}}
\hspace{.0in}
\begin{tabular}{|c|c|}
\hline
1 & 4 \\ \hline
2 \kern-.25pt 3 & 5 \kern-.25pt 6 \\ \hline
\end{tabular}
\end{subfigure}

\caption{Transforming an $r$-tuple $(T_1,\hdots,T_r)$ of set-valued tableaux into a single set-valued tableau of density $\rho(k,r)$ via horizontal concatenation, alongside the inverse procedure.}
\label{fig: tableaux concatenation example}
\end{figure}

For an example illustrating both directions of the bijection from Theorem \ref{thm: Raney numbers}, see Figure \ref{fig: tableaux concatenation example}.  Notice that our interpretation of $R_{k,r}(n)$ as the cardinality of $\svt((n+1)^2,\rho(k,r))$ immediately recovers the $k$-Catalan specialization $R_{k,k}(n-1) = C_n^k$ of Heubach, Li and Mansour \cite{HLM} when $r=k$.  Also notice the special meaning of Theorem \ref{thm: Raney numbers} as it applies to the extreme case of $r=1$, as set-valued tableaux of density $\rho(k,1)$ have an empty cell at position $(2,1)$.  In this case, one may construct a bijection from $\svt((n+1)^2,\rho(k,1))$ to $\svt(n^2,\rho(k,k))$ by deleting the first column of $T \in \svt((n+1)^2,\rho(k,1))$ and re-indexing the remaining $nk$ entries of $T$ by $x \mapsto (x-1)$.  This bijection directly corresponds to the Raney number identity $R_{k,1}(n) = R_{k,k}(n-1) = C_n^k$.

\subsection{The $(s,t)$-Tennis Ball Problem}
\label{subsec: tennis ball problem}

The so-called ``tennis ball problem" was introduced by Tymoczko and Henle in their logic textbook \cite{TH}, and was subsequently formalized by Mallows and Shapiro \cite{MS}.  The classic version of the problem begins with $2n$ tennis balls, numbered $1,2,\hdots,2n$, and proceeds through $n$ turns.  For the first turn, one takes the balls numbers $1$ and $2$ and randomly throws one of them out of your window onto their lawn.  During the $i^{th}$ turn, the balls numbered $2i-1$ and $2i$ are added to the $i-1$ balls leftover from previous steps, and one of those $i+1$ balls is thrown onto your lawn.  The problem then asks how many different sets of balls are possible on one's lawn after $n$ steps.  Independent from Mallows and Shapiro \cite{MS}, Grimaldi and Moser \cite{GM} proved that the number of such arrangements was the Catalan number $C_{n+1}$.

Merlini, Sprugnoli, and Verri \cite{MSV} generalized these phenomena to the $(s,t)$-tennis ball problem, whereby $s$ new balls are added and $t$ balls are thrown onto the lawn during each turn.  If we let $\mathcal{B}_{s,t}(n)$ denote the number of arrangements possible after $n$ turns in the generalized problem, Merlini et al. \cite{MSV} showed that $\mathcal{B}_{k,1}(n) = C^k_{n+1}$.  Generating functions for all $\mathcal{B}_{s,t}(n)$ were later developed by de Mier and Noy \cite{MN}, who placed the resulting arrangements in bijection with certain classes of N-E lattice paths (see Section \ref{sec: lattice paths}).

All of this is relevant in that the $\mathcal{B}_{s,t}(n)$ admits a straightforward combinatorial interpretation in terms of Young tableaux, an interpretation that has yet to appear anywhere in the literature.  As seen with the row-constant densities of Theorem \ref{thm: tennis ball problem}, the $\mathcal{B}_{s,t}(n)$ represent a one-parameter generalization of the $k$-Catalan numbers that are distinct from the Raney numbers.

\begin{theorem}
\label{thm: tennis ball problem}
Fix $s,t \geq 1$ such that $s \geq t$.  The solution to the $(s,t)$-tennis ball problem after $n$ turns is $\mathcal{B}_{s,t}(n) = \vert \svt((n+1)^2,\rho) \vert$, where $\rho$ is the row-constant density $\rho_{1,j} = t$, $\rho_{2,j} = s-t$.
\end{theorem}
\begin{proof}
We define a bijection from the set of arrangements after $n$ turns to $\svt((n+1)^2,\rho)$.  Place the $nt$ numbers corresponding to balls on the lawn in increasing order across the first row of the Young diagram of shape $\lambda = (n+1)^2$, beginning with the cell at $(1,2)$ and ensuring each cell receives $t$ integers.  Then place the remaining $n(s-t)$ integers in increasing order across the first $n$ second row cells of that same Young diagram, ensuring each cell receives $s-t$ integers.  After reindexing entries by $x \mapsto x+t$, place the integers $1,\hdots,t$ in the cell at position $(1,1)$, and place the integers $ns-s+t+1,\hdots,ns$ at position $(2,n+1)$.  The resulting set-valued tableau is row-standard by construction.  To see that it is column-standard, notice that the first $i$ turns of the procedure collectively involve throwing $ti$ balls with labels at most equal to $si$.  This means that, before reindexing, the largest entry at position $(1,i+1)$ is at most $si$.  It also implies that, before reindexing, the largest entry at position $(2,i)$ is at least $si$.  It follows that the smallest entry at $(2,i+1)$, before reindexing, is at least $si+1$.
\end{proof}

Comparing Theorem \ref{thm: tennis ball problem} with the results of Subsection \ref{subsec: Raney numbers}, in the case of $t=1$ we directly recover the result of Merlini, Sprungnoli, and Verri \cite{MSV} that $\mathcal{B}_{s,t}(n) = C_{n+1}^s$.  Also note that the result of Theorem \ref{thm: tennis ball problem} may be further generalized to the ``non-constant" tennis ball problem of de Mier and Noy \cite{MN}.  If $\vec{s} = \lbrace s_i \rbrace$ and $\vec{t} = \lbrace t_i \rbrace$ are sequences of positive integers such that $t_i < s_i$ for all $i$, the $(\vec{s},\vec{t})$-tennis ball problem is the generalization of the tennis ball problem wherein $s_i$ new balls are added and $t_i$ balls are thrown out the window during the $i^{th}$ turn.  Equivalent reasoning to Theorem \ref{thm: tennis ball problem}, yields the following combinatorial interpretation of the solution $\mathcal{B}_{\vec{s},\vec{t}}(n)$ to this fully-generalized problem.

\begin{theorem}
\label{thm: tennis ball problem 2}
Let $\vec{s} = \lbrace s_i \rbrace$ and $\vec{t} = \lbrace t_i \rbrace$ be sequences of positive integers such that $t_i < s_i$ for all $i$.  Then the solution to the $(\vec{s},\vec{t})$-tennis ball problem after $n$ turns is $\mathcal{B}_{\vec{s},\vec{t}}(n) = \vert \svt((n+1)^2,\rho) \vert$, where $\rho$ is shown below.

\begin{center}
\setlength{\tabcolsep}{4pt}
\begin{tabular}{|c|c|c|c|c|}
\hline
$(1)$ & $(t_1)$ & $\hdots$ & $(t_{n-1})$ & $(t_n)$ \\ \hline
$(s_1 - t_1)$ & $(s_2 - t_2)$ & $\hdots$ & $(s_n - t_n)$ & $(1)$ \\ \hline
\end{tabular}
\end{center}
\end{theorem}

\section{Enumeration of Two-Row Set-Valued Tableaux}
\label{sec: enumeration}

Although an enumeration of $\svt(\lambda,\rho)$ for general $\lambda$ and $\rho$ isn't currently tractable, the two-row case of $\lambda = (n_1,n_2)$ is sufficiently simple that methodologies may be developed for arbitrary $\rho$.  In this section, we present a technique for such enumerations that we refer to as ``density shifting".  This procedure sets up a bijection between $\svt(\lambda,\rho)$ and a collection of sets $\svt(\lambda',\rho'_i)$, where $\lambda'=(n_1 - 1, n_2 - 1)$ and the varying densities $\rho'_i$ are determined by $\rho$.

To define our procedure of \textbf{density shifting}, fix $\lambda = (n_1,n_2)$ and a density $\rho$ with $\rho_{1,j} = a_j$, $\rho_{2,j} = b_j$.  We focus on the first two columns of arbitrary $T \in \svt(\lambda,\rho)$, and consider the relationship of the integers $\beta_1 < \hdots < \beta_{b_1}$ at position $(2,1)$ to the integers $\alpha_1 < \hdots < \alpha_{a_2}$ at position $(1,2)$.  In particular, we identify the smallest integer $\beta_m$ such that $\beta_m > \alpha_{a_2}$.  The integers $\beta_m,\beta_{i+1},\hdots,\beta_{b_1}$ may then be moved to the cell at $(2,2)$ without violating column-standardness, where they are necessarily the $b_1 - m + 1$ smallest integers at $(2,2)$.  The remaining integers $\beta_1,\hdots,\beta_{i-1}$ at position $(2,1)$ are then moved to the cell at $(1,2)$, where they are smaller than $\alpha_{a_2}$ but their relationship to $\alpha_1,\hdots,\alpha_{a_2 - 1}$ depends upon the choice of $T$.  With the cell at $(2,1)$ empty, the entire first row of the tableau is deleted and the remaining entries are re-indexed according to $x \mapsto x - a_1$.  This produces a tableau $d(T) \in \svt(\lambda',\rho')$ for $\lambda'=(n_1-1,n_2-1)$ and some $\rho'$ with $\rho'_{i,j} = \rho_{i+1,j+1}$ for $j >1$ and first row densities $\rho'_{1,j}$ determined by $T$.  We refer to this new tableau $d(T)$ as the \textbf{density shift} of $T$.  See Figure \ref{fig: density shifting example} for an example.

\begin{figure}[ht!]
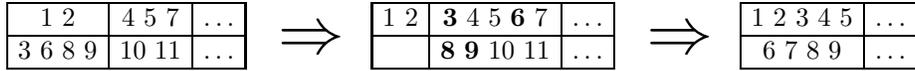

\centering
\setlength{\tabcolsep}{4pt}
\begin{tabular}{|c|c|c|}
\hline
1 2 & 4 5 7 & $\hdots$ \\ \hline
3 6 8 9 & 10 11 & $\hdots$ \\ \hline
\end{tabular}
\hspace{.075in}
\raisebox{-4pt}{\scalebox{2.5}{$\Rightarrow$}}
\hspace{.05in}
\begin{tabular}{|c|c|c|}
\hline
1 2 & \textbf{3} 4 5 \textbf{6} 7 & $\hdots$ \\ \hline
 & \textbf{8} \textbf{9} 10 11 & $\hdots$ \\ \hline
\end{tabular}
\hspace{.075in}
\raisebox{-4pt}{\scalebox{2.5}{$\Rightarrow$}}
\hspace{.05in}
\begin{tabular}{|c|c|}
\hline
1 2 3 4 5 & $\hdots$ \\ \hline
6 7 8 9 & $\hdots$ \\ \hline
\end{tabular}
\caption{A two-row set-valued tableau $T$ and its density shift $d(T)$.}
\label{fig: density shifting example}
\end{figure}

The map $T \mapsto d(T)$ is well-defined into $\bigcup_i \svt(\lambda',\rho'_i)$, assuming that one appropriately determines the collection of shifted densities $\rho'_i$.  However, the map is far from injective, as $d(T)$ does not remember which of its (non-maximal) entries at position $(1,1)$ were shifted to that position.  Relating $\vert \svt(\lambda,\rho) \vert$ to the $\vert \svt(\lambda',\rho_i') \vert$ requires that we account for all possible positioning of shifted entries at $(2,1)$.

In the statement of Theorem \ref{thm: density shifting} and all that follows, we use a Young diagram of shape $\lambda$ labelled with cell densities $\rho_{i,j}$ from $\rho$ to denote the cardinality $\svt(\lambda,\rho)$.

\begin{theorem}
\label{thm: density shifting}
For any two-row shape $\lambda = (n_1,n_2)$ and density $\rho$ as shown,

$$
\vert \svt(\lambda,\rho) \vert \ = \ \setlength{\tabcolsep}{4pt}
\begin{tabular}{|>{$}c<{$}|>{$}c<{$}|>{$}c<{$}|>{$}c<{$}|}
\hline
(a_1) & (a_2) & (a_3) & \hdots \\ \hline
(b_1) & (b_2) & (b_3) & \hdots \\ \hline
\end{tabular}
\ = \ \sum_{i=0}^{b_1} \binom{a_2 + i - 1}{i}
\begin{tabular}{|>{$}c<{$}|>{$}c<{$}|>{$}c<{$}|}
\hline
(a_2) & (a_3) & \hdots \\ \hline
(b_1 + b_2 - i) & (b_3) & \hdots \\ \hline
\end{tabular}
$$
where the sets $\svt(\lambda',\rho'_i)$ inside the summation on the right are all of shape $\lambda = (n_1-1,n_2-1)$.
\end{theorem}
\begin{proof}
For fixed $\lambda = (n_1,n_2)$ and $\rho$ as shown, partition $\svt(\lambda,\rho)$ into subsets $S_1,\hdots,S_{b_1}$ where $S_i = \lbrace T \in \svt(\lambda,\rho) \ \vert \ \text{precisely $i$ entries at $(2,1)$ smaller than largest entry at $(1,2)$} \rbrace$.  When restricted to a specific $S_i$, $T \mapsto d(T)$ defines a function $d_i: \svt(\lambda,\rho) \rightarrow \svt(\lambda',\rho')$ with $\lambda'=(n_1-1,n_2-1)$ and $\rho'_{1,1} = a_2 - i$, $\rho'_{2,1}=b_2 + b_1 - i$.  For any $0 \leq i \leq b_1$, we claim $d_i$ is onto and is exactly $m$-to-$1$, where $m = \binom{a_2 + i - 1}{i}$.

So consider $\svt(\lambda',\rho')$, and notice that the upper-leftmost cell of any $T' \in \svt(\lambda',\rho')$ is filled with the integers $\lbrace 1,2,\hdots,a_2+i \rbrace$.  For any choice $\vec{u}$ of $i$ integers from $[a_2 + i - 1]$, define a map $d_i^{-1}: \svt(\lambda',\rho') \rightarrow \svt(\lambda,\rho)$ as follows:

\begin{enumerate}
\item For any $T' \in \svt(\lambda',\rho')$, remove the $i$ integers at position $(1,1)$ corresponding to $\vec{u}$ as well as the $b_1 - i$ smallest integers at position $(2,1)$.
\item Append a new column to the left of the tableau and populate the bottom cell of that new column with the $b_1$ integers removed during Step \#1.
\item Reindex all entries in the resulting (partially-filled) tableau by $x \mapsto x + a_1$ and add the set $[a_1]$ to the top cell in the new leftmost column, resulting in $d_i^{-1}(T) \in \svt(\lambda,\rho)$.
\end{enumerate}

The map $d_i^{-1}$ has been defined so that $d \circ d_i^{-1}(T') = T'$ for every $T' \in \svt(\lambda',\rho')$.  If we once again let $\beta_1 < \hdots < \beta_{b_1}$ denote the integers at $(2,1)$ of some $T \in S_i$, we have $d_i^{-1} \circ d_i (T) = T$ for precisely those $T$ where $\beta_1,\hdots,\beta_i$ correspond to $\vec{u}$ in $d_i(T)$.  If we let $S_i \vert_{\vec{u}}$ denote the subset of $S_i$ with this restriction upon the $\beta_1,\hdots,\beta_i$, it follows that $S_i \vert_{\vec{u}}$ is in bijection with $\svt(\lambda',\rho')$.  Ranging over all choices of $\vec{u}$ then allows us to conclude that $\vert S_i \vert = \binom{a_2 + i - 1}{i} \vert \svt(\lambda',\rho') \vert$.

To obtain the summation from the theorem, note that changing the density $\rho'_{1,1}$ at position $(1,1)$ has no effect on the size of the sets $\svt(\lambda,\rho)$.  We are then allowed to assume that $\rho'_{1,1} = a_2 = \rho_{1,2}$ is unchanged as we pass from $\rho$ to $\rho'$.  As the $S_i$ partition $\svt(\lambda,\rho)$, varying $1 \leq i \leq b_1$ then yields the required summation.
\end{proof}

Now consider a pair of $n$-tuples of non-negative integers $\vec{x} = (x_1,\hdots,x_n)$ and $\vec{y} = (y_1,\hdots,y_n)$.  One may define a dominance ordering on these tuples whereby $\vec{x} \preceq \vec{y}$ if $x_1 + \hdots + x_i \leq y_1 + \hdots y_i$ for every $1 \leq i \leq n$.  Using this notation, Theorem \ref{thm: density shifting} may be repeatedly applied to derive the following.

\begin{theorem}
\label{thm: repeated density shifting}
For any two-row shape $\lambda = (n_1,n_2)$ and density $\rho$ as shown,
$$
\vert \svt(\lambda,\rho) \vert \ = \ \setlength{\tabcolsep}{4pt}
\begin{tabular}{|>{$}c<{$}|>{$}c<{$}|>{$}c<{$}|>{$}c<{$}|}
\hline
(a_1) & (a_2) & (a_3) & \hdots \\ \hline
(b_1) & (b_2) & (b_3) & \hdots \\ \hline
\end{tabular}
\ = \
\sum_{(i_1,\hdots,i_{n_1-1}) \preceq (b_1,\hdots,b_{n_1-1})}
\prod_{j=1}^{n_1-1}\binom{a_{j+1} + i_j - 1}{i_j}
$$
\end{theorem}
\begin{proof}
Repeated application of Theorem \ref{thm: density shifting} immediately yields
$$
\setlength{\tabcolsep}{4pt}
\begin{tabular}{|>{$}c<{$}|>{$}c<{$}|>{$}c<{$}|>{$}c<{$}|} \hline
(a_1) & (a_2) & (a_3) & \hdots \\ \hline
(b_1) & (b_2) & (b_3) & \hdots \\ \hline
\end{tabular}
\ = \
\sum_{i_1=0}^{b_1} \binom{a_2 + i_1 - 1}{i_1}
\sum_{i_2=0}^{b_1+b_2-i_1} \binom{a_3 + i_2 - 1}{i_2}
\sum_{i_3=0}^{b_1+b_2+b_3-i_1-i_2} \binom{a_4 + i_3 - 1}{i_3} \hdots
$$
If we assume that $i_1,i_2,\hdots$ must be positive integers, the later summations on the right side are equivalent to $i_1 + i_2 \leq b_1 + b_2$, $i_1 + i_2 + i_3 \leq b_1 + b_2 + b_3$, etc.
\end{proof}

Observe that the equation of Theorem \ref{thm: repeated density shifting} involves all cell densities apart from $a_1$ and $b_n$, aligning with our intuition that changing $a_1$ or $b_n$ does not effect $\vert \svt(\lambda,\rho) \vert$.  Also note that applying Theorems \ref{thm: density shifting} and \ref{thm: repeated density shifting} to non-rectangular shapes $\lambda=(n_1,n_2)$ merely requires that we set $b_j = 0$ for every $j > n_2$. 

\begin{example}
For $\lambda = (n,n)$ and $a_j = 1$, $b_j = k-1$ for all $j$, the product of Theorem \ref{thm: repeated density shifting} becomes
$$\prod_{j=1}^{n_1-1}\binom{a_{j+1} + i_j - 1}{i_j} \ = \prod_{j=1}^{n_1-1} \binom{i_j}{i_j} \ = \ 1$$
Thus $\vert \svt(\lambda,\rho) \vert$ is the number of $(n-1)$-tuples of non-negative integers $(i_1,\hdots,i_{n-1}) \preceq (k-1,\hdots,k-1)$.  If we let $y = kn - i_1 - \hdots i_{n-1}$, these tuples may be placed in bijection with the set $\mathcal{D}_n^k$ of $k$-good paths by $(i_1,\hdots,i_{n-1}) \mapsto E N^{i_1} E N^{i_2} \hdots E N^{i_{n-1}} E N^y$.  Hence $\vert \svt(\lambda,\rho) \vert = C^k_n$, as expected from Section \ref{sec: generalized Catalan numbers}.
\end{example}

\begin{example}
More generally, for $\lambda= (n,n)$ and any density with $a_j = 1$ for all $j$, Theorem \ref{thm: repeated density shifting} shows that $\vert \svt(\lambda,\rho) \vert$ equals the number of $(n-1)$-tuples of non-negative integers $(i_1,\hdots,i_{n-1}) \preceq (b_1,\hdots,b_{n-1})$.  These tuples may be placed in bijection with the set of N-E lattice paths from $(0,0)$ to $(n,b_1 + \hdots + b_n)$ that lie weakly below the path $P = E N^{b_1} E N^{b_2} \hdots E N^{b_n}$ via the same map as the previous example.  See Section \ref{sec: lattice paths} for a further generalization of this result.
\end{example}

\section{Set-Valued Tableaux \& Two-Dimensional Lattice Paths}
\label{sec: lattice paths}

We close by drawing a bijection between arbitrary $\svt(\lambda,\rho)$ with $\lambda = n^2$ and various classes of two-dimensional lattice paths, generalizing the phenomenon exemplified in Figure \ref{fig: k-Catalan interpretations}.  This requires a consideration of all integer lattice paths from $(0,0)$ to $(a,b)$ that use only East $E=(1,0)$ and North $N=(0,1)$ steps, which we refer to as N-E lattice paths of shape $(a,b)$.  To avoid ambiguities in the definition of our lattice paths, for the rest of this section we restrict our attention to densities $\rho$ that lack cells of density zero.  It is straightforward to extend all of the results below to $\rho$ with zero density cells, so long as there does not exist $j$ where $\rho_{2,j}=\rho_{1,j+1}=0$.

So fix $\lambda = n^2$ and consider the density $\rho$ where $\rho_{1,j} = a_j,\rho_{2,j}=b_j$.  If $\sum_j a_j = a$ and $\sum_j b_j = b$, there exists a map $\psi_\rho: \svt(\lambda,\rho) \rightarrow \mathcal{P}$ into the set $\mathcal{P}$
of lattice paths of shape $(a,b)$ such that first row entries of $T \in \svt(\lambda,\rho)$ correspond to East steps in $\psi_\rho(T)$ and second row entries of $T$ correspond to North steps in $\psi_\rho(T)$.  The map $\psi_\rho$ is always an injection, but its image is dependent upon the choice of $\rho$. 

To characterize $\im(\psi_\rho)$, we introduce a partial order on $\mathcal{P}$.  For $P_1,P_2 \in \mathcal{P}$, define $P_1 \geq P_2$ if $P_1$ lies weakly above $P_2$ across $0 \leq x \leq a$.\footnote{This poset is isomorphic to Young's lattice via the map that takes a path to the Young diagram lying above its conjugate.}  Our map $\psi_\rho$ respects this partial order in the following sense.

\begin{lemma}
\label{thm: tableaux vs lattice paths lemma} 
For fixed $\lambda = n^2$ and $\rho$, take $P_1,P_2 \in \mathcal{P}$ such that $P_1 \geq P_2$.  If $P_1 \in \im(\psi_\rho)$, then $P_2 \in \im(\psi_\rho)$.
\end{lemma}
\begin{proof}
We prove the statement for when $P_1$ directly covers $P_2$.  This corresponds to the situation where $P_2$ may be obtained from $P_1$ by replacing a single $NE$ subsequence with an $EN$ subsequence at the same position.  Assume that this $NE \mapsto EN$ replacement occurs at the $i$ and $i+1$ steps of both $P_1$ and $P_2$.  For $T_1 \in \svt(\lambda,\rho)$ with $\psi_\rho(T_1) = P_1$, the integer $i$ must appear in the second row of $T_1$ and $i+1$ must appear in the first row of $T_1$.  Column-standardness of $T_1$ implies that $i$ must appear in a more leftward column of $T_1$ than does $i+1$.  Then define $T_2 \in \svt(\lambda,\rho)$ to be the tableau obtained by flipping the positions of $i$ and $i+1$ in $T_1$.  As $i$ and $i+1$ are consecutive integers and since $i$ originally appeared left of $i+1$ in $T_1$, $T_2$ is row- and column-standard.  By construction, $\psi_\rho(T_2) = P_2$.
\end{proof}

For any two-row density $\rho$ with non-zero cell densities, there exists unique $T_{max} \in \svt(\lambda,\rho)$ such that, for all $j$, every integer in the $j^{th}$ column of $T_{max}$ is smaller than every integer in the $(j+1)^{st}$ column of $T_{max}$.  This is precisely the tableau such that $\psi_\rho(T_{max}) = E^{a_1} N^{b_1} \hdots E^{a_n} N^{b_n}$.  For any such $\rho$, the order ideal generated by $\psi_\rho(T_{max})$ will precisely correspond to $\im(\psi_\rho)$.

\begin{theorem}
\label{thm: tableaux vs lattice paths}
Fix $\lambda = n^2$ and density $\rho$ with $\rho_{1_j}=a_j > 0$ $\rho_{2,j}=b_j > 0$ for all $j$.  If we define $P_{max} \in \mathcal{P}$ by $P_{max} = E^{a_1} N^{b_1} \hdots E^{a_n} N^{b_n}$, then $\svt(\lambda,\rho)$ is in bijection with $I = \lbrace P \in \mathcal{P} \ \vert \ P \leq P_{max} \rbrace$.
\end{theorem}
\begin{proof}
As $P_{max} = \psi_\rho(T_{max})$, $P_{max} \in \im(\psi_\rho)$ and Lemma \ref{thm: tableaux vs lattice paths lemma} immediately gives $I \subseteq \im(\psi_\rho)$.  Since $\psi_\rho$ is known to be injective, it is only left to show that $\im(\psi_\rho) \subseteq I$.

So assume by contradiction there exists $T \in \svt(\lambda,\rho)$ with $\psi_\rho(T) \nleq P_{max}$.  There exists a smallest index $i$ such that the $i^{th}$ steps of both $\psi_\rho(T)$ and $P_{max}$ begin at the same point, the $i^{th}$ step of $\psi_\rho(T)$ is a $N$ step, and the $i^{th}$ step of $P_{max}$ is an $E$ step.  This means that the subtableaux of $T$ and $T_{max}$ containing only $\lbrace 1,\hdots,i-1 \rbrace$ must have the same shape and density, while the integer $i$ lies in the first row of $T_{max}$ but in the second row of $T$.  If $i$ lies at position $(1,j)$ of $T$, the construction of $T_{max}$ then implies that every integer in the $(j-1)^{st}$ columns of both $T$ and $T_{max}$ is smaller than $i$.  It follows that $i$ must lie in the $(2,j)$ cell of $T$.  Yet then there must exist an entry at position $(1,j)$ of $T$ that is larger than $i$, implying that $T$ is not column-standard.
\end{proof}

\begin{example}
\label{ex: lattice path bijection, basic}
For $\lambda = n^2$ and row-constant density $\rho$ with $\rho_{1,j} = 1, \rho_{1,j} = k-1$, $\psi_\rho(T_{max}) = (E N^{k-1})^n$.  N-E lattice paths lying weakly below $\psi_\rho(T_{max})$ are in bijection with N-E lattice paths lying weakly below the line $y = (k-1)x$, recovering the bijection of Section \ref{sec: generalized Catalan numbers} between $\mathcal{D}_n^k$ and $\svt(n^2,\rho)$.
\end{example}

\begin{example}
\label{ex: lattice path bijection, tennis ball}
For $\lambda = (n+1)^2$ and row-constant density $\rho$ with $\rho_{1,j} = t$, $\rho_{2,j}=s-t$, $\psi_\rho(T_{max}) = (E^t N^{s-t})^{n-1}$.  This recovers the bijection between $\mathcal{B}_{s,t}(n)$ and N-E lattice paths utilized by de Mier and Noy \cite{MN}. 
\end{example}

As a more involved example, we use Theorem \ref{thm: tableaux vs lattice paths} to derive a new combinatorial interpretation of the rational Catalan numbers in terms of standard set-valued Young tableaux.  For relatively prime positive integers $a$ and $b$, there exists a rational Catalan number $C(a,b) = \frac{1}{a+b} \binom{a+b}{a}$.  As originally shown by Bizley \cite{Bizley} and extended by Grossman \cite{Grossman}, the rational Catalan numbers equal the number of rational Dyck paths of shape $(a,b)$, by which we mean N-E lattice paths from $(0,0)$ to $(a,b)$ that lie weakly below the line of rational slope $y = \frac{b}{a}x$.

Applying Theorem \ref{thm: tableaux vs lattice paths} merely requires the identification of a unique maximal lattice path $P_{max}$ among the set of all rational Dyck paths of shape $(a,b)$.  The path $P_{(a,b)} = E^1 N^{c_1} \hdots E^1 N^{c_a}$ with $c_i = \floor{\frac{bi}{a}} - \floor{\frac{b(i-1)}{a}}$ satisfies this condition, as $\sum_{i=1}^k c_i = \floor{bk}{a}$ for all $1 \leq k \leq a$ and $P_{(a,b)}$ has a Northwest corner at the first integer lattice point below the intersection of $y = \frac{b}{a}x$ with $x=k$ for every $1 \leq k \leq a$.

\begin{corollary}
\label{thm: rational Catalan numbers}
Take positive integers $a,b$ such that $\gcd(a,b) = 1$.  Then $\vert \svt (a^2,\rho) \vert = C(a,b) = \frac{1}{a+b} \binom{a+b}{a}$ for the density $\rho$ with $\rho_{1,j} = 1$, $\rho_{2,j} = \floor{\frac{bj}{a}} - \floor{\frac{b(j-1)}{a}}$.
\end{corollary}

\end{document}